\documentclass[reqno]{amsart}

\usepackage{url,fancyvrb,amscd,verbatim,amssymb,numberedblock}
\usepackage[usenames,dvipsnames]{color}

\definecolor{darkgreen}{rgb}{0,0.5,0}
\usepackage[pagebackref=true, colorlinks, citecolor=darkgreen]{hyperref}

\newcommand{\NN}{\mathbf{N}} 
\newcommand{\ZZ}{\mathbf{Z}} 
\newcommand{\QQ}{\mathbf{Q}} 
\newcommand{\FF}{\mathbf{F}} 
\newcommand{\PP}{\mathbf{P}} 
\renewcommand{\P}{\mathbf{P}} 

\providecommand{\Hrig}{H_{\text{rig}}} 	      	
\providecommand{\BigOh}{O}          		
\providecommand{\SoftOh}{\widetilde{O}} 		

\DeclareMathOperator{\GL}{GL}
\DeclareMathOperator{\Frob}{F}          
\DeclareMathOperator{\Gal}{Gal} 	
\DeclareMathOperator{\ord}{ord} 	

\newtheorem{thm}{Theorem}[section]

\newtheorem{prop}[thm]{Proposition}

\newtheorem{defn}[thm]{Definition}
\newtheorem{exmp}[thm]{Example}
\newtheorem{rem}[thm]{Remark}

\newtheorem{assump}[thm]{Assumption}
\newtheorem{alg}[thm]{Algorithm}

\renewcommand*{\backref}[1]{}
\renewcommand*{\backrefalt}[4]{%
  \ifcase #1 %
    \relax
  \or
    $\uparrow$#2.%
  \else
    $\uparrow$#2.%
  \fi%
}

\newlength{\proofmargin}
\title{Explicit Coleman integration for curves}
\author{Jennifer S. Balakrishnan}
\address{Jennifer S. Balakrishnan, Department of Mathematics and Statistics, Boston University, 111 Cummington Mall, Boston, MA 02215, USA}
\email{jbala@bu.edu}
\author{Jan Tuitman}
\address{Jan Tuitman, KU Leuven,
         Departement Wiskunde,
         Celestijnenlaan 200B,
         3001 Leuven,
         Belgium}
\email{jan\_tuitman@hotmail.com}


\begin{document}
\date{\today}

\begin{abstract} The Coleman integral is a $p$-adic line integral 
that plays a key role in computing several important 
invariants in arithmetic geometry. We give an algorithm for explicit Coleman 
integration on curves, using the algorithms of the second author~\cite{tuitman:p1, tuitman:pc-general} 
 to compute the action of Frobenius on $p$-adic cohomology. 
We present a collection of examples computed with our implementation. This includes integrals on a genus 55 curve, where other methods do not currently seem practical.
\end{abstract}

\maketitle

\section{Introduction}
In a series of papers in the 1980s, Coleman formulated a $p$-adic theory of line integration on curves and higher-dimensional varieties with good reduction at $p$
and gave numerous applications in arithmetic geometry. This includes the computation of $p$-adic polylogarithms~\cite{coleman:dilogarithms}, 
torsion points on Jacobians of curves~\cite{coleman:torsion}, rational points on certain curves with small Mordell-Weil rank~\cite{coleman:chabauty}, 
$p$-adic heights on curves (joint with Gross)~\cite{coleman-gross}, and $p$-adic regulators in $K$-theory (joint with de Shalit)~\cite{coleman-deshalit}. 
In \cite{coleman-deshalit}, Coleman and de Shalit  also introduced a theory of {iterated} $p$-adic integration on curves, which plays an important role in Kim's 
nonabelian Chabauty program \cite{kim:chabauty} to compute rational points on curves.

Besser and de Jeu~\cite{besser-dejeu} gave the first algorithm for explicit computation of these integrals---now known as \emph{Coleman integrals}---in the case of iterated 
integrals on $\P^1\setminus\{0,1,\infty\}$. These integrals compute $p$-adic polylogarithms, which are conjecturally related to special values of $p$-adic $L$-functions. 
Balakrishnan, Bradshaw and Kedlaya~\cite{bbk} gave an algorithm to compute single 
Coleman integrals on odd degree models of hyperelliptic curves, which was further generalized to iterated Coleman integrals on arbitrary
hyperelliptic curves in~\cite{balakrishnan:iterated, balakrishnan:even}. These algorithms all rely on an algorithm for computing the action of Frobenius on 
$p$-adic cohomology to realize Dwork's principle of \emph{analytic continuation along Frobenius}. In the case of odd degree hyperelliptic 
curves, this is achieved by Kedlaya's algorithm~\cite{kedlaya:mw}.

Because of all of the applications mentioned above, it is of interest to develop practical 
algorithms to carry out Coleman integration on \emph{any} smooth curve. 
In the present work, we do this by building on work of 
Tuitman~\cite{tuitman:p1, tuitman:pc-general}, which generalizes Kedlaya's algorithm to this setting. We give algorithms to 
compute single Coleman integrals on curves and develop the precision bounds necessary to obtain provably
correct results. Moreover, we provide a complete implementation~\cite{colemangit} of our algorithms in the computer algebra system \texttt{Magma}~\cite{magma} and present a selection of examples, including the computation of  torsion points on Jacobians and carrying out the Chabauty--Coleman method for finding rational points on curves.  We also compare our algorithms to other leading techniques. We present a selection of examples computed using our algorithm, including integrals on a genus 55 curve, where other techniques do not currently seem practical. Our computation shows that the Jacobian of this curve has positive Mordell--Weil rank. The case of iterated Coleman integrals will be discussed in a subsequent paper. 

The structure of the paper is as follows: 
First, in Section \ref{sec:padiccoh} we recall what we need from the theory of $p$-adic cohomology and the algorithms from~\cite{tuitman:p1,tuitman:pc-general}. 
In Section~\ref{sec:integrals}, we present our algorithms for Coleman integration. Next, in Section~\ref{sec:precision}, we discuss the precision bounds necessary to obtain provably correct results.  In Section~\ref{sec:complexity}, we carry out a complexity analysis of our algorithm and compare it with other approaches.
Finally, in Section~\ref{sec:examples}, we conclude with a collection of examples computed with our implementation~\cite{colemangit}.

\section{$p$-adic cohomology}\label{sec:padiccoh}

Let $X$ be a nonsingular projective geometrically irreducible curve of genus~$g$ over $\QQ$ given by a (possibly singular) 
plane model $Q(x,y)=0$ with $Q(x,y) \in \ZZ[x,y]$ a polynomial that is irreducible and monic in $y$. Recall that such a model 
can easily be obtained from other representations of $X$, e.g. by computing (a defining equation of) its function field. 

Let $d_x$ and $d_y$ denote the degrees of the morphisms $x$ and $y$, respectively, from $X$ to the projective line. Note
that these correspond to the degrees of $Q$ in the variables $y$ and $x$, respectively. For the performance of our algorithms
it will be best to first take $d_x$ as small as possible (ideally equal to the gonality of the curve) 
and then $d_y$ as small as possible for this value of $d_x$.

\begin{defn} \label{defn:Delta}
Let $\Delta(x) \in \ZZ[x]$ denote the discriminant of $Q$ with respect to the variable $y$.
Moreover, define $r(x) \in \ZZ[x]$ to be the squarefree polynomial with the same zeros as $\Delta(x)$,
in other words, $r=\Delta/(\gcd(\Delta,\frac{d\Delta}{dx}))$.
\end{defn}

\begin{defn}\label{defn:intbases} Let $W^0 \in \GL_{d_x}(\QQ[x,1/r])$ and
$W^{\infty} \in \GL_{d_x}(\QQ[x,1/x,1/r])$ denote matrices such that, if we denote 
\[ b^0_j = \sum_{i=0}^{d_x-1} W^0_{i+1,j+1} y^i \; \; \; \; \mbox{ and } \; \; \; \; b^{\infty}_j = \sum_{i=0}^{d_x-1} W^{\infty}_{i+1,j+1} y^i \] 
for all $0 \leq j \leq d_x-1$, then
\begin{enumerate}
\item $[b^{0 \;}_0,\ldots,b^{0 \;}_{d_x-1}]$ is an integral basis for $\QQ(X)$ over $\QQ[x]$,
\item $[b^{\infty}_0,\ldots,b^{\infty}_{d_x-1}]$ is an integral basis for $\QQ(X)$ over $\QQ[1/x]$,
\end{enumerate}
where $\QQ(X)$ denotes the function field of $X$. Moreover, let $W \in \GL_{d_x}(\QQ[x,1/x])$ denote
the change of basis matrix $W=(W^0)^{-1} W^{\infty}$.
\end{defn}

There are good algorithms available to compute such matrices, e.g. \cite{hess,bauch}. 

\begin{rem}We assume that $X$ is a curve over $\QQ$ since it is more delicate to compute integral bases in function fields over a $p$-adic
field, both in practice and in theory (to finite precision everything is smooth).\end{rem}

\begin{exmp}Let $X/\QQ$ be an odd degree monic hyperelliptic curve of genus $g$ given by the plane model $$Q(x,y) = y^2 - f(x)=0.$$
We have that $$r(x) = f(x)$$ and
$$
W^0 = \begin{pmatrix}
      1  & 0 \\
      0  & 1
      \end{pmatrix},
\; \; \; \; \; \;
W^{\infty} = \begin{pmatrix}
             1 & 0 \\
             0 & 1/x^{g+1}
             \end{pmatrix}.
$$ 
This means that $b^0 = [1, y]$ and $b^{\infty} = [1,y/x^{g+1}]$ are integral bases for the function field of $X$ over 
$\QQ[x]$ and $\QQ[{1/x}]$, respectively.
\end{exmp}

\begin{defn}
We say that the triple $(Q,W^0,W^{\infty})$ has good reduction at a prime number~$p$ if the conditions
below (taken from \cite[Assumption 1]{tuitman:pc-general}) are satisfied. \end{defn}

\begin{assump}[{\cite[Assumption 1]{tuitman:pc-general}}]\label{tuitman1} \mbox{ }
\begin{enumerate}
\item The discriminant of $r(x)$ is contained in $\ZZ_p^{\times}$.
\item If we denote 
$b^0_j = \sum_{i=0}^{d_x-1} W^0_{i+1,j+1} y^i$ and 
$b^{\infty}_j = \sum_{i=0}^{d_x-1} W^{\infty}_{i+1,j+1} y^i$ for all $0 \leq j \leq d_x-1$, and if we 
let $\FF_p(x,y)$ be the field of fractions of $\FF_p[x,y]/(Q)$, then:
\begin{enumerate}
\item The reduction modulo~$p$ of $[b^{0 \;}_0,\ldots,b^{0 \;}_{d_x-1}]$ is an integral basis for $\FF_p(x,y)$ over $\FF_p[x]$.
\item The reduction modulo~$p$ of $[b^{\infty}_0,\ldots,b^{\infty}_{d_x-1}]$ is an integral basis for $\FF_p(x,y)$ over $\FF_p[1/x]$.
\end{enumerate}
\item $W^0 \in \GL_{d_x}(\ZZ_p[x,1/r])$ and $W^{\infty} \in \GL_{d_x}(\ZZ_p[x,1/x,1/r])$.
\item Denote:
\begin{align*}
\mathcal{R}^0        &= \ZZ_p[x]b^{0}_0 \; \; \; \; \; +\ldots+\ZZ_p[x]b^{0}_{d_x-1}, \\
\mathcal{R}^{\infty} &= \ZZ_p[1/x]b^{\infty}_0+\ldots+\ZZ_p[1/x]b^{\infty}_{d_x-1}.
\end{align*}
For a ring $R$, let $R_{\textrm{red}}$ denote the reduced ring obtained by quotienting out by the nilradical. Then the discriminants of the finite $\ZZ_p$-algebras $(\mathcal{R}^0/(r(x)))_{\textrm{red}}$ and
$(\mathcal{R}^{\infty}/(1/x))_{\textrm{red}}$ are contained in $\ZZ_p^{\times}$.\end{enumerate}
\end{assump}

\begin{rem} 
These conditions imply that the curve $X$ has good reduction at~$p$ but are stronger. 
Note that any triple $(Q,W^0,W^{\infty})$ has good reduction at all but finitely many
prime numbers~$p$. 
\end{rem}

From now on, we fix a prime~$p$ and a triple $(Q,W^0,W^{\infty})$ which has good reduction at~$p$. 
In \cite[Proposition 2.3]{tuitman:pc-general}) it is explained how one can associate to this data
a smooth curve $\mathcal{X}$ over $\ZZ_p$ such that $\mathcal{X}\otimes \QQ_p \cong X \otimes \QQ_p$. 
Let $X^{an}$ denote the rigid analytic space over $\QQ_p$ which is the generic fibre of $\mathcal{X}$.
There is a specialization map from $X^{an}$ to the reduction of $X$ modulo~$p$. The fibres of this map 
are called \emph{residue disks}. (For further background on rigid analytic geometry, see 
\cite{fresnel-vanderput}.)

\begin{defn}
We say that a point of $X^{an}$ is
\emph{very infinite} if its $x$-coordinate is $\infty$ and \emph{very bad} if
it is either very infinite or its $x$-coordinate is a zero of $r(x)$. \end{defn}

\begin{rem}From the fact
that $(Q,W^0,W^{\infty})$ has good reduction at~$p$, it follows that a residue disk
contains at most one very bad point and that this point is defined over an unramified
extension of $\QQ_p$.  \end{rem}

For a very
bad point $P$, we will denote the ramification index of the map $x$ by $e_P$.  We let $U$ denote the complement of the very bad points in $X^{an}$. 

\begin{defn}
We say that a residue
disk (as well as any point inside it) is \emph{infinite} or \emph{bad} if it contains 
a very infinite or a very bad point, respectively. A point or residue disk is called \emph{finite} if 
it is not infinite and \emph{good} if it is not bad. \end{defn}

\begin{rem}Note that all infinite points and infinite 
residue disks are bad.\end{rem}

\begin{rem}
If a point is very bad, this can mean one of three things: 
\begin{enumerate}
\item $x(P) = \infty$,
\item the fiber of $X$ above $x(P)$ contains a ramification point,
\item the fiber of $X$ above $x(P)$ contains a point mapping to a singularity of the plane model $Q(x,y)=0$. 
\end{enumerate}
\end{rem}

We now introduce the main rings over which we work. Let $\langle \rangle^\dag$ denote the ring of overconvergent functions obtained by weak completion of the corresponding
polynomial ring. 

\begin{defn} We denote 
\begin{align*}
S &=\ZZ_p \langle x, 1/r \rangle,         &S^{\dag} &=\ZZ_p \langle x, 1/r \rangle^{\dag}, \\
R &= \ZZ_p \langle x, 1/r, y \rangle/(Q), &R^{\dag} &=\ZZ_p \langle x, 1/r, y \rangle^{\dag}/(Q).
\end{align*}
A Frobenius lift  $\Frob_p:R^{\dag} \rightarrow R^{\dag}$ is defined as a continuous ring homomorphism that 
reduces to the $p$-th power Frobenius map modulo $p$. 
\end{defn}

\begin{thm} \label{thm:froblift} There exists a Frobenius lift $\Frob_p: R^{\dag} \rightarrow R^{\dag}$ 
for which $\Frob_p(x)=x^p$. 
\end{thm}

\begin{proof}
See \cite[Thm. 2.6]{tuitman:pc-general}.
\end{proof}

\begin{defn}
For a point $P$ on a smooth curve, we let $\ord_P$ denote the corresponding discrete valuation
on the function field of the curve. In particular, $\ord_0$ and $\ord_{\infty}$ are
the discrete valuations on the rational function field $\QQ(x)$ corresponding to the points $0$ and 
$\infty$ on $\PP^1_{\QQ}$. We extend these definitions to matrices by taking the minimum over 
their entries.
\end{defn}

From the assumption that $(Q,W^0,W^{\infty})$ has good reduction at~$p$, it follows that the rigid
cohomology spaces $\Hrig^1(U \otimes \QQ_p)$ and $\Hrig^1(X \otimes \QQ_p)$ are isomorphic to their 
algebraic de Rham counterparts~\cite{baldachiar}. 

\begin{defn}\label{def:basis}
Let $[\omega_1,\ldots,\omega_{2g}]$ be $p$-adically integral $1$-forms on $U$ such that
\begin{enumerate}
\item $\omega_1,\ldots,\omega_{g \;}$ form a basis for $H^0(X,\Omega^1_X)$,
\item $\omega_1,\ldots,\omega_{2g}$ form a basis for $\Hrig^1(X \otimes \QQ_p)$,
\item $\ord_P(\omega_i) \geq -1$ for all~$i$ at all finite very bad points $P$,
\item $\ord_P(\omega_i) \geq -1+(\ord_0(W)+1)e_P$ for all~$i$ at all very infinite points $P$.
\end{enumerate}
\end{defn}
In~\cite{tuitman:p1,tuitman:pc-general}, it is explained how $1$-forms satisfying properties (2)-(4) can be computed.
The algorithm can be easily adapted so that (1) is satisfied as well, which is the convention we take. 

\begin{defn} \label{defn:frobdecomp}
The $p$-th power Frobenius $\Frob_p$ acts on $\Hrig^1(X \otimes \QQ_p)$, so there exist a
matrix $\Phi \in M_{2g \times 2g}(\QQ_p)$ and functions $f_1,\ldots,f_{2g} \in R^{\dag} \otimes \QQ_p$
such that 
\[
\Frob_p^*(\omega_i) = df_i + \sum_{j=1}^{2g} \Phi_{ij} \omega_j  
\]
for $i=1, \ldots, 2g$.
\end{defn}

Let us briefly recall from \cite{tuitman:p1, tuitman:pc-general} how the matrix~$\Phi$ and the functions~$f_1,\ldots,f_{2g}$ are computed. 

\begin{alg} \mbox{ } \label{alg:pcc}
\begin{enumerate} 
\item \textbf{Compute the Frobenius lift.} Determine $\Frob_p$ as in Theorem~\ref{thm:froblift}, i.e. set $\Frob_p(x)=x^p$ and
determine the elements $\Frob_p(1/r) \in S^{\dag}$ and $\Frob_p(y) \in R^{\dag}$ by Hensel lifting.  
\item \textbf{Finite pole order reduction.} For $i=1,\ldots,2g$, find $f_{i,0} \in R^{\dag} \otimes \QQ_p$ 
such that
\[
\Frob_p^*(\omega_i) = df_{i,0} + G_i \left( \frac{dx}{r(x)} \right),
\]
where $G_i \in R \otimes \QQ_p$ only has poles at very infinite points. 
\item \textbf{Infinite pole order reduction.} For $i=1,\ldots,2g$, find $f_{i,\infty} \in R \otimes \QQ_p$  such that
\[
\Frob_p^*(\omega_i) = df_{i,0} + df_{i,\infty} + H_i \left(\frac{dx}{r(x)}\right),
\]
where $H_i \in R \otimes \QQ_p$ still only has poles at very infinite points $P$ and satisfies
$$\ord_{P}(H_i) \geq (\ord_0(W)-\deg(r)+2)e_P$$ 
at all these points. 
\item \textbf{Final reduction.} For $i=1,\ldots,2g$, find $f_{i,end} \in R \otimes \QQ_p$ 
such that
\[
\Frob_p^*(\omega_i) = df_{i,0} + df_{i,\infty} + df_{i,end} + \sum_{j=1}^{2g} \Phi_{ij} \omega_j, 
\]
where $\Phi \in M_{2g \times 2g}(\QQ_p)$ is the matrix of $\Frob_p^*$ on $\Hrig^1(U \otimes \QQ_p)$ with respect
to the basis $[\omega_1,\ldots,\omega_{2g}]$. \\
\end{enumerate}
\end{alg}

The matrix $\Phi$ and the functions 
$f_i=f_{i,0}+f_{i,\infty}+f_{i,end}$
are exactly what we need from~\cite{tuitman:p1,tuitman:pc-general} to compute Coleman integrals.

\section{Coleman integrals}\label{sec:integrals}

Let $K/\QQ_p$ be a totally ramified extension. Our goal is to compute the Coleman integral 
$\int_{P}^{Q} \omega$ of a $1$-form $\omega \in \Omega^1(U \otimes \QQ_p)$ of the second kind  
between points $P,Q \in X(K)$. 

The Coleman integral satisfies several key properties, which we will use throughout our integration algorithms:
\begin{thm}[Coleman, Coleman--de Shalit]\label{prop_coleman_int} Let $\eta, \xi$ be $1$-forms on a wide open $V$ of $X^{an}$ and $P,Q,R \in V(K)$. Let $a,b \in K$.  The definite Coleman integral has the following properties:
\begin{enumerate}\item Linearity: $\int_P^Q (a\eta + b\xi) = a \int_P^Q \eta + b\int_P^Q \xi.$
\item Additivity in endpoints: $\int_P^Q \xi = \int_P^R \xi + \int_R^Q \xi.$
\item Change of variables: If $V' \subset X'$ is a wide open subspace of a rigid analytic space $X'$ and $\phi: V \rightarrow V'$ a rigid analytic map then $\int_P^Q \phi^* \xi = \int_{\phi(P)}^{\phi(Q)} \xi.$
\item Fundamental theorem of calculus: $\int_P^Q df = f(Q) - f(P)$ for $f$ a rigid analytic function on $V$.
\item Galois equivariance: the integral is compatible with the action of $\Gal(K/\QQ_p)$.
\end{enumerate}
\end{thm}
\begin{proof} 
\cite{coleman:torsion} for $1$-forms of the second kind and~\cite{coleman-deshalit} for general $1$-forms.
\end{proof}

Let us first explain how we specify a point $P$. Note that giving $(x,y)$-coordinates might not be 
sufficient even for a finite very bad point, since there may be multiple points on 
$X$ lying above a singular point $(x,y)$ of the plane model defined by $Q$. However, a point $P$ is determined 
uniquely by the value of $x$ ($1/x$ if $P$ is infinite) together with the values of the functions $b^0$ 
($b^{\infty}$ if $P$ is infinite).  Note that all of these values are $p$-adically integral.
In our implementation, we therefore specify a point $P$ by storing three values $(\texttt{P`x,P`b,P`inf})$:
\begin{enumerate}
\item \texttt{P`x}: the $x$-coordinate of $P$ ($1/x$ if $x$ is infinite),
\item \texttt{P`b}: the values of the functions $b^0$ ($b^{\infty}$ if $P$ is infinite),
\item \texttt{P`inf}: true or false, depending on whether the point $P$ is infinite or not.
\end{enumerate}

We will often need power series expansions of functions in terms of a \emph{local coordinate} (i.e., a uniformizing parameter) 
$t$ at $P$. This local coordinate should not just 
be a local coordinate at~$P$ on $X \otimes \QQ_p$, but on the model $\mathcal{X}$ over 
$\ZZ_p$ obtained from the triple $(Q,W^0,W^{\infty})$ as in~\cite[Prop. 2.3]{tuitman:pc-general}. Then it follows 
that the reduction modulo~$p$ of $t$ is a local coordinate at the reduction modulo~$p$ of $P$ and that the residue 
disk at $P$ is given by $\lvert t \rvert<1$. In a bad residue disk, we will always expand functions at the very bad
point. Therefore, in the following proposition, we only consider points that are either good or very bad.
\begin{prop} \label{prop:locparam} Let $P \in X(\QQ_p)$ be a point. Assume that $P$ is either good
or very bad. As a local coordinate at $P$, we can take
\begin{equation*}
t = \begin{cases} x-x(P) &\;\textrm{if}\; e_P=1\; (\textrm{or}\; t=1/x \;\textrm{if}\; P \;\textrm{is infinite}),\\
 b^0_i-b^0_i(P) \;\textrm{for some}~i &\;\textrm{otherwise}\; (\textrm{with}\; b^0 \;\textrm{replaced by}\; b^{\infty} \;\textrm{if}\; P \;\textrm{is infinite}).
\end{cases}
\end{equation*}
\end{prop}

\begin{proof} 
By definition $e_P=\ord_P(x-x(P))$ (or $e_p=\ord_P(1/x$) if $P$ is infinite). So if $e_P=1$ then 
$t=(x-x(P))$ (or $t=1/x$ if $P$ is infinite) is a local coordinate at $P$ on $X \otimes \QQ_p$.
If $e_P \geq 2$, then at least one of the $b^0_i-b^0_i(P)$ (with $b^0$ replaced by $b^{\infty}$ if 
$P$ is infinite) has to be a local coordinate at $P$ on  $X \otimes \QQ_p$, since otherwise there
would be no functions on $X \otimes \QQ_p$ of order~$1$ at~$P$. In both cases, since
$(Q,W^0,W^{\infty})$ has good reduction at~$p$, the divisor defined by $t$ on $\mathcal{X}$ is
smooth over $\ZZ_p$, so that $t$ is also a local coordinate at~$P$ in the stronger sense explained above.
\end{proof}

After choosing a local coordinate~$t$ at~$P$, in our implementation we compute \texttt{xt,bt} where
\begin{enumerate}\item \texttt{xt} is the power series expansion in $t$ of the function $x$ ($1/x$ if $P$ is infinite),
\item \texttt{bt} is the tuple of power series expansions in $t$ of the functions $b^0$ ($b^{\infty}$ if $P$ is infinite).
\end{enumerate}
Note that all of these power series have $p$-adically integral coefficients.
From \texttt{xt,bt} we will be able to determine the power series expansion in $t$ of any function which is regular at $P$.

A $1$-form $\omega \in \Omega^1(U \otimes \QQ_p)$ is of the form $f dx$ with $f \in R \otimes \QQ_p$. We will usually represent it as follows:
\begin{eqnarray*} \label{eqn:omega}
\omega = \sum_{i=0}^{d_x-1} \sum_{j \in J} \frac{f_{ij}(x)}{r(x)^j} b^0_i \frac{dx}{r} 
\end{eqnarray*}
with $f_{ij} \in \QQ_p[x]$ such that  $\deg(f_{ij}) < \deg(r(x))$ for all $i,j$, since $\omega$ needs
to be in this form to start the cohomological reduction procedures outlined in Section~\ref{sec:padiccoh}. \\

We begin by describing the computation of \emph{tiny} integrals. 

\begin{defn}A tiny integral $\int_P^Q \omega$ is a Coleman integral with endpoints $P,Q \in X(\QQ_p)$ that lie in the same residue disk.
\end{defn}

\begin{alg}[Computing the tiny integral $\int_P^Q \omega$]\label{alg:tiny} \mbox{ }
\begin{enumerate} 
\item If the residue disk of $P,Q$ is bad, then find the very bad point $P'\in X(\QQ_p)$, otherwise set $P'=P$.
\item Compute a local coordinate $t$ and the power series expansions \emph{\texttt{xt,bt}} at $P'$. 
\item Integrate using $t$ as coordinate: 
$$\int_P^Q \omega = \int_{t(P)}^{t(Q)} \omega(t).$$ 
The Laurent series expansion $\omega(t)$ can be determined from \emph{\texttt{xt,bt}}. Note that $\omega$ is of the second kind, 
so the coefficient of $t^{-1}dt$ is zero.
\end{enumerate} 
\end{alg}

\begin{rem}
The calculation of tiny integrals does not require 
computing the action of Frobenius on the cohomology space $\Hrig^1(X \otimes \QQ_p)$. 
This can be a useful consistency check for the integration algorithms that follow, which do use the 
computation of the action of Frobenius. 
\end{rem}

\begin{rem}
Note that Algorithm~\ref{alg:tiny} can also be applied for points defined over $\QQ_p(p^{1/e})$ for some positive integer~$e$ (as long as they are within the same residue disk). 
This is useful for applications in bad residue disks, as we will see later (Algorithm \ref{alg:colemani}).
\end{rem}

When $P,Q \in X(\QQ_p)$ do not lie in the same residue disk, this approach breaks down since the Laurent series expansions
do not converge anymore. In this case we will compute the Coleman integrals $ \int_{P}^Q \omega_i$ for $i=1,\ldots,2g$ 
by solving a linear system imposed by the $p$-th power Frobenius map $F_p$. We first assume that the functions $f_1,\ldots,f_{2g}$ from 
Section~\ref{sec:padiccoh} converge at $P,Q$. Note that 
$f_1,\ldots,f_{2g}$ converge at all good points, but only at bad points that are not too close to the corresponding very bad point.
This will be made more precise in the next section.

\begin{alg}[Compute the $\int_P^Q \omega_i$ assuming $f_1,\ldots,f_{2g}$ converge at $P,Q$]\label{alg:colemangoodi} \mbox{ } 
\begin{enumerate}
\item Compute the action of Frobenius on $\Hrig^1(X \otimes \QQ_p)$ using Algorithm \ref{alg:pcc} and store $\Phi$ and $f_1,\ldots,f_{2g}$.
\item Determine the tiny integrals $\int_P^{F_p(P)} \omega_i$ and $\int_{F_p(Q)}^Q \omega_i$ for $i=1,\ldots,2g$ using Algorithm \ref{alg:tiny}.
\item Compute $f_i(P) - f_i(Q)$ for $i = 1, \ldots, 2g$ and use the system of equations  
\begin{equation*}
\sum_{j=1}^{2g} (\Phi-I)_{ij} \left( \int_P^Q \omega_j \right) = f_i(P)-f_i(Q) - \int_P^{F_p(P)} \omega_i - \int_{F_p(Q)}^Q \omega_i
\end{equation*} to solve for all $\int_P^Q \omega_i$, as in \cite[Algorithm 11]{bbk}.
\end{enumerate} 
\end{alg}

\begin{rem}
Note that the matrix $\Phi-I$ is invertible, since the eigenvalues of $\Phi$ are algebraic numbers of complex absolute value $p^{1/2}$. 
\end{rem}

\begin{rem}The algorithm above follows from the first four properties of the Coleman integral in Theorem \ref{prop_coleman_int}, and in particular, change of variables is carried out via Frobenius, which is a rigid analytic map.\end{rem}

\begin{rem}An alternate approach to Algorithm \ref{alg:colemangoodi} that applies outside of the bad residue disks is to compute Teichm\"{u}ller points (fixed points of the Frobenius map) within the residue disks, solve the resulting linear system between Teichm\"{u}ller points, then correct endpoints via tiny integrals.\end{rem}

\begin{rem}Note that Algorithm~\ref{alg:colemangoodi} can also be applied for points defined over $\QQ_p(p^{1/e})$ for some positive integer~$e$. This is useful for applications in bad residue disks, as we will see below in Algorithm \ref{alg:colemani}.\end{rem}

When $P$ or $Q$ are bad points and $f_1,\ldots,f_{2g}$ do not converge there, the idea is simply to find points 
$P',Q'$ in the residue disks of $P$ and $Q$ where these functions do converge, compute the integrals between 
the new points, and correct for the difference with tiny integrals.

\begin{alg}[Computing the $\int_P^Q \omega_i$ in general]\label{alg:colemani} \mbox{ } 
\begin{enumerate}
\item Determine $P',Q'$ in the residue disks of $P,Q$ at which all functions $f_1,\ldots,f_{2g}$ converge. (See Remark \ref{rem:nearbdprec}.)
\item Compute the tiny integrals $\int_P^{P'} \omega_i$ and $\int_{Q'}^Q \omega_i$ for $i=1,\ldots,2g$ using Algorithm~\ref{alg:tiny}.
\item Determine $\int_{P'}^{Q'} \omega_i$ for $i=1,\ldots,2g$ using Algorithm~\ref{alg:colemangoodi}.
\item Compute $$\int_P^Q \omega_i = \int_P^{P'} \omega_i + \int_{P'}^{Q'} \omega_i +\int_{Q'}^Q \omega_i.$$
\end{enumerate} 
\end{alg}

In general, we have to take the points $P',Q'$ to be defined over some (totally ramified) 
extension $K$ of $\QQ_p$ to get far enough away from the very bad point 
in the bad residue disk; see Remark \ref{rem:nearbdprec}. We will always take this extension to be of the form
$\QQ_p(p^{1/e})$ for some positive integer~$e$. Recall that Algorithms~\ref{alg:tiny} and~\ref{alg:colemangoodi} 
can still be applied in this case and that we may take $P' \in X(\QQ_p)$ in Algorithm~\ref{alg:tiny}. Since computing in 
extensions is more expensive, integrals involving bad points are usually the hardest to compute. \\

For more general $1$-forms of the second kind $\omega \in \Omega^1(U \otimes \QQ_p)$, we can now compute the Coleman integrals $\int_{P}^Q \omega$ as follows
from the output of Algorithms~\ref{alg:colemangoodi} and~\ref{alg:colemani}.
\begin{alg}[Computing $\int_P^Q \omega$] \label{alg:colemangood} \mbox{ }
\begin{enumerate}
\item Use Steps (2),(3) and (4) of Algorithm~\ref{alg:pcc} to find $f \in R$ and $c_i \in \QQ_p$ for $i=1,\ldots,2g$ such that $$\omega = df+ \sum_{i=1}^{2g} c_i \omega_i.$$ 
\item Compute $f(Q)-f(P)$ and determine
\[
\int_{P}^{Q} \omega = f(Q)-f(P)  + \sum_{i=1}^{2g} c_i \int_{P}^Q \omega_i. 
\]
\end{enumerate}
\end{alg}

\begin{rem}
Note that we are only considering points $P,Q$ defined over a  totally ramified extension $K$ of $\QQ_p$ because we want the residue field to be 
$\FF_p$ so that we work with a lift of $p$-power Frobenius.  It would be of interest to extend our work to points defined over arbitrary 
finite extensions of $\QQ_p$ as discussed in~\cite[Remark 12]{bbk} and more generally work with a lift of $q$-power Frobenius.
\end{rem}

\section{Precision bounds} \label{sec:precision}

So far we have not paid any attention to the fact that we can only compute to finite $p$-adic and $t$-adic
precision in our algorithms. By \emph{precision} we will always mean \emph{absolute} $p$-adic precision, i.e., 
the valuation of the error term. We extend the $p$-adic valuation and the notion of precision to all 
finite extensions of $\QQ_p$, where they will take non-integer values in general. 

Let us start with tiny integrals.

\begin{prop} \label{prop:prectiny} Let $e$ be a positive integer and $P,Q \in X(\QQ_p(p^{1/e}))$ two points in 
the same residue disk accurate to precision $N$. Let $t$ be a local coordinate (in the sense of 
Proposition~\ref{prop:locparam}) at the point $P'$ from Algorithm~\ref{alg:tiny}. 
Suppose that $\omega = g(t) dt$ is a differential of the second kind with 
$$g(t) = a_{-k} t^{-k} + a_{-k+1} t^{-k+1} + \ldots  \in \ZZ_p[[t]][t^{-1}]$$
for some positive integer $k$. If $g$ is accurate to $p$-adic precision $N$ and truncated modulo $t^l$, then the tiny integral $\int_P^Q \omega$ 
computed as in Algorithm~\ref{alg:tiny} is correct to precision $\min\{\nu_1,\nu_2,\nu_3\}$ where:
\begin{align*}
\nu_1 &= 1/e +  \min_{i \geq l}   \{ i/e-\lfloor \log_p(i+1) \rfloor \}, \\
\nu_2 &= N   +  \min_{0 \leq i \leq l-1}   \{ i/e- \lfloor \log_p(i+1) \rfloor \}, \\
\nu_3 &= N   -  k\max \{ \ord_{p}(t(P)),\ord_p(t(Q))\}-\lfloor \log_p(k-1) \rfloor. 
\end{align*}
\end{prop}
\begin{proof} Recall from Algorithm~\ref{alg:tiny} that
\[
\int_{P}^Q \omega = \int_{t(P)}^{t(Q)} \omega(t) = \sum_{i=-k}^{\infty} \frac{a_i}{i+1} \left( t(Q)^{i+1} - t(P)^{i+1} \right).
\]
where $a_{-1}=0$ since $\omega$ is of the second kind. Since $P,Q$ both lie in the residue disk given by $\lvert t \rvert <1$,
we have that $\ord_p(t(P)),\ord_p(t(Q)) \geq 1/e$.

First, we bound the error introduced by omitting the terms with $i \geq l$. Note that 
$$\ord_{p}(t(P)^{i+1}),\ord_p(t(Q)^{i+1}) \geq (i+1)/e$$ 
and 
$$\ord_{p}(i+1) \leq \lfloor \log_p(i+1) \rfloor.$$ 
Therefore, the valuation of this error is at least $\nu_1$.

Next, we consider the error coming from terms with $0 \leq i \leq l-1$. Since $t(P),t(Q)$ are accurate to precision $N$ 
and have valuation at least~$1/e$, we have that $t(P)^{i+1},t(Q)^{i+1}$ are correct to precision $N+i/e$. Therefore, the valuation of
this error is at least $\nu_2$.

Finally, we  bound the error coming from terms with $-k \leq i \leq -2$. This time $t(P)^{i+1},t(Q)^{i+1}$ are correct to precision at least
$N+i\ord_{p}(t(P)),N+i\ord_{p}(t(Q))$, respectively (since the loss of precision of an inversion is $2$ times the valuation). Therefore, the valuation 
of the error is at least $\nu_3$ this time.
\end{proof}

\begin{rem} \label{rem:l}
Since we always have that $\nu_2 \leq N$, there is no point in increasing the $t$-adic precision $l$ further if 
$\nu_1 \geq N$ already. Therefore, in our implementation we take $l$ to be minimal such that $\nu_1 \geq N$. 
\end{rem}

To compute integrals that are not tiny, in Algorithm~\ref{alg:colemangoodi} we have to evaluate the functions 
$$f_i=f_{i,0}+f_{i,\infty}+f_{i,end}$$ from Section~\ref{sec:padiccoh} at the endpoints for $i=1,\ldots,2g$. 
Evaluating an element of 
$R^{\dag}~\otimes~\QQ_p$ at a bad point leads to problems with convergence and loss of precision. We first recall 
from~\cite{tuitman:p1,tuitman:pc-general} what we know about the poles of the functions
$f_{i,0},f_{i,\infty},f_{i,end} \in R^{\dag} \otimes \QQ_p$.

The only poles of infinite order are those of the $f_{i,0}$ at the finite very bad points. 
It follows from~\cite[Prop. 2.12, Prop. 3.3, Prop. 3.7]{tuitman:pc-general} that
\begin{equation} \label{eq:f0i} f_{i,0} = \sum_{j=0}^{d_x-1} \sum_{k=1}^{\infty} \frac{c_{ijk}(x)}{r(x)^k} b^0_j,\end{equation}
for all $i$, where the $c_{ijk}$ are elements of $\QQ_p[x]$ of degree smaller than $\deg(r)$ that satisfy
\begin{equation} \label{eq:cijk} \ord_p(c_{ijk}) \geq \lfloor k/p \rfloor +1 - \lfloor \log_p(k e_0) \rfloor\end{equation}
with $e_0 = \max\{e_P \colon P \mbox{ finite very bad point}\}$. 

Similarly, it follows from~\cite[Prop. 2.12, Prop. 3.4, Thm. 3.6]{tuitman:pc-general} that
\begin{equation} \label{eq:finftyi}
f_{i,\infty} = \sum_{j=0}^{d_x-1} \sum_{k=0}^{\kappa_1} c_{ijk} x^k b^0_j = \sum_{j=0}^{d_x-1} \sum_{k=\kappa_2}^{\kappa_3} d_{ijk} x^k b^{\infty}_j
\end{equation}
for all $i$, where the $c_{ijk},d_{ijk}$ are elements of $\QQ_p$ and 
$$\kappa_3 \leq -\min\{p (\ord_0(W)+1) ,(\ord_{\infty}(W^{-1})+1)\},$$ where $W=(W^0)^{-1} W^{\infty}$.
This determines bounds on $\kappa_1,\kappa_2$ as well.

Finally, it follows from~\cite[Thm. 3.6]{tuitman:pc-general} that
\begin{equation} \label{eq:fendi}
f_{i,end} = \sum_{j=0}^{d_x-1} \sum_{k=0}^{\lambda_1} c_{ijk} x^k b^0_j  = \sum_{j=0}^{d_x-1} \sum_{k=\lambda_2}^{\lambda_3} d_{ijk} x^k b^{\infty}_j
\end{equation}
for all $i$, where the $c_{ijk}, d_{ijk}$ are elements of $\QQ_p$ and 
\[
\lambda_3 \leq -(\ord_0(W)+1).
\]
Note that this determines bounds on $\lambda_1,\lambda_2$ as well.

\begin{prop}
On a finite bad residue disk, the functions $f_{i,0}$ converge outside of the 
closed disk defined by $\ord_p(r(x)) \geq 1/p$.
\end{prop}

\begin{proof}
This is clear from \eqref{eq:f0i} and \eqref{eq:cijk}. 
\end{proof}

\begin{rem}\label{rem:nearbdprec}
Let $t$ denote a local coordinate at the very bad point of a finite bad residue disk.
Then we have that $\ord_p(r(x)) < 1/p$ is equivalent to the condition $\ord_p(t) < \frac{1}{pe_P}$. 
Consequently, for the functions $f_{i,0}$ to converge at a point $P' \in X(\QQ_p(p^{1/e}))$ 
in the residue disk of $P$, we need to take $e > pe_P$. 
\end{rem}

When $f_1,\ldots,f_{2g}$ do converge at a point $P$, their computed values at this point will
suffer some loss of $p$-adic precision in general. In the next three propositions we quantify 
this precision loss for good, finite bad, and infinite points, respectively.

\begin{prop} \label{prop:precgood}
Suppose that the functions $f_{i,0},f_{i,\infty},f_{i,end}$ are accurate to 
precision~$N$. Moreover, let $e$ be a positive integer and let
$P \in X(\QQ_p(p^{1/e}))$ be a good point that is accurate to precision $N$. 
Then the computed values $f_i(P)$ are correct to precision $N$ as well.
\end{prop}
\begin{proof} 
Note that a good point is always finite. Since we have that $\ord_p(x(P)) \geq 0$ and $\ord_p(r(x(P)))=0$,
there is no loss of precision in evaluating~\eqref{eq:f0i} and the expressions in the middle
of~\eqref{eq:finftyi} and~\eqref{eq:fendi}.
\end{proof}

\begin{prop} \label{prop:precfinitebad}
Suppose that the functions $f_{i,0},f_{i,\infty},f_{i,end}$ are accurate to precision $N$. Moreover, 
let $e$ be a positive integer and let $P \in X(\QQ_p(p^{1/e}))$ be a finite bad point that is accurate to precision $N$. 
Let $\epsilon=\ord_p(r(P))$ and suppose that $\epsilon < 1/p$. Define a function $\pi$ on positive integers by
\[
\pi(k) =  \max \{N, \lfloor k/p \rfloor +1 - \lfloor \log_p(k e_0) \rfloor \},
\]
where $e_0 = \max\{e_P \colon P \mbox{ finite bad point }\}$.
Then the computed values $f_i(P)$ are correct to precision
$$\min_{k \in \NN} \{ \pi(k) - k \epsilon \}.$$
\end{prop}
\begin{proof} 
In this case $\ord_p(x(P)) \geq 0$, but $\ord_p(r(x(P))) = \epsilon$ with $0 < \epsilon < 1/p$.
Clearly there is still no loss of precision in evaluating the expressions in the middle
of~\eqref{eq:finftyi} and~\eqref{eq:fendi}. However for the $f_{i,0}$ there will be
loss of precision. After dropping the terms with valuation greater than or equal to~$N$
in~\eqref{eq:f0i}, the coefficient $c_{ijk}$ will be correct to precision $\pi(k)$ for all $k$. 
Dividing by $r(x(P))^k$ leads to the loss of $k \epsilon$ 
digits of precision, so the terms corresponding to $k$ will be correct to precision $\pi(k)-k\epsilon$.
Taking the minimum over $k$, we obtain the result. 
\end{proof}

\begin{prop} \label{prop:precinfinite}
Suppose that the functions $f_{i,0},f_{i,\infty},f_{i,end}$ are accurate to precision $N$. 
Moreover, let $e$ be a positive integer and let
$P \in X(\QQ_p(p^{1/e}))$ be an infinite point that is accurate to precision $N$.
Let $\epsilon=\ord_p(1/x(P))$. Then the computed values $f_i(P)$ are correct to precision
\[
N + \epsilon \min \{ \ord_{\infty}(W^{-1})+1, p (\ord_0(W)+1) \}. 
\]
\end{prop}
\begin{proof} 
In this case $\ord_p(x(P)) = - \epsilon < 0$ and $\ord_p(r(x(P))) = - \deg(r) \epsilon$. Let us first consider the $f_{i,0}$. 
Determining the $b^0_j(P)$ from the $b^{\infty}_j(P)$ in $\eqref{eq:f0i}$ leads to a precision loss of $-\ord_{\infty}(W^{-1}) \epsilon $.
However, since $\deg(c_{ijk}) < \deg(r)$, we have that
$$\ord_p \left( \frac{c_{ijk}(x(P))}{r(x(P))^k} \right) \geq \epsilon $$
for all $k\geq 1$. Therefore, we recover precision $\epsilon$ and the loss of precision will be at most $-(\ord_{\infty}(W^{-1})+1)\epsilon$. 
Evaluating the expressions on the right of~\eqref{eq:finftyi} and~\eqref{eq:fendi} leads to precision loss at most 
$$-\min\{p (\ord_0(W)+1) ,(\ord_{\infty}(W^{-1})+1)\} \epsilon$$
and 
$$-(\ord_0(W)+1) \epsilon,$$ respectively. The result follows easily from this.
\end{proof}

Now all that is left to analyze in Algorithm~\ref{alg:colemangoodi} is the precision loss from solving the linear system, 
i.e. computing the matrix $(\Phi-I)^{-1}$ and multiplying by it.

\begin{prop}
Suppose that the matrix $\Phi$ is $p$-adically integral and accurate to precision $N$. 
Moreover, let $e$ be a positive integer and let $P,Q \in X(\QQ_p(p^{1/e}))$ be points accurate to precision $N$. 
Suppose that the right hand side of (3) in Algorithm~\ref{alg:colemangoodi} is 
accurate to precision $N' \leq N$ according to 
Propositions~\ref{prop:prectiny},~\ref{prop:precgood},~\ref{prop:precfinitebad}, and~\ref{prop:precinfinite}.
Then the integrals $\int_{P}^Q \omega_i$ as computed in Algorithm~\ref{alg:colemangoodi} are correct to precision
\[
N' - \ord_p(\det(\Phi-I)). 
\]
\end{prop}
\begin{proof}
This follows since $(\Phi-I)^{-1}$ has valuation at least $-\ord_p(\det(\Phi-I))$ and is 
correct to precision $N-\ord_p(\det(\Phi-I))$.
\end{proof}

\begin{rem}
If we do not assume that $\Phi$ is $p$-adically integral, then we can show
that the integrals $\int_{P}^Q \omega_i$ as computed in Algorithm~\ref{alg:colemangoodi} 
are correct to precision $$N' - \ord_p(\det(\Phi-I))-\delta$$ 
with $\delta$ defined as in~\cite[Definition 4.4]{tuitman:pc-general}.
\end{rem}

\begin{rem}
To analyze the loss of precision in Algorithm~\ref{alg:colemangood}, we proceed as follows. First, we use~\cite[Prop. 3.7, Prop. 3.8]{tuitman:pc-general} 
to determine the precision to which $f$ and the $c_i$ are correct. Then we proceed as in Propositions~\ref{prop:precgood},
~\ref{prop:precfinitebad}, and~\ref{prop:precinfinite} to determine the precision of the computed values of $f(P)$, $f(Q)$ and $\int_{P}^Q \omega_i$
for $i=1,\ldots,2g$. Finally, we determine the precision to which $\int_{P}^Q \omega$ is correct, taking into account the valuations of the $c_i$
as well. 
\end{rem}

\section{Complexity analysis and comparison with other algorithms}\label{sec:complexity}
In this section, we discuss the complexity of our algorithm and compare it to other approaches. We use the
$\SoftOh(-)$ notation that ignores logarithmic factors, i.e. $\SoftOh(f)$ denotes the class of functions that
lie in $\BigOh(f \log^k(f))$ for some $k \in \NN$. To be able to apply the complexity analysis 
from~\cite{tuitman:pc-general} we will need one more assumption from that paper:

\begin{assump}[{\cite[Assumption 2]{tuitman:pc-general}}]\label{tuitman2} \mbox{ }
Both $-\ord_P(W^0)$ and $-\ord_P(W^{\infty})$ are contained in $\BigOh(d_x d_y)$ for all $P \in \PP^1(\overline{\QQ})$.
\end{assump}

In~\cite{tuitman:pc-general}, it is explained why this is a reasonable assumption: for instance, standard algorithms for
computing integrals bases of function fields will yield matrices $W^0, W^{\infty}$ satisfying this condition.

\subsection{Complexity analysis}\label{subsec:complexity}\mbox{} \\

\begin{prop} \label{prop:complexityfrob}
Let notation and assumptions be as introduced in Section~\ref{sec:padiccoh}. The matrix $\Phi$ and the functions $f_1, \ldots, f_{2g}$ 
from Definition~\ref{defn:frobdecomp} can be computed to $p$-adic precision~$N$ using Algorithm~\ref{alg:pcc} in time 
$\SoftOh(p d_x^4 d_y^2 (N^2 + d_x d_y N))$. 
\end{prop}

\begin{proof}
Take the maximum over the four steps in~\cite[Section 4]{tuitman:pc-general}, leaving~$N$ instead of replacing it with
$\BigOh(d_x d_y)$. Note that technically, here we have to replace~$N$ by the working precision of the algorithm
from \cite{tuitman:pc-general}, necessary to obtain the matrix $\Phi$ to precision $N$. However, by the argument 
from~\cite[Proposition 4.9]{tuitman:pc-general}, this working precision can be chosen to be $N + \BigOh(\log{(d_x d_y)})$, 
yielding the same expression for the complexity since $\log(d_x)$ and $\log(d_y)$ are absorbed by the $\SoftOh$ symbol.
\end{proof}
In what follows, we will restrict to the (generic) case of integrals between good points,
only making a few remarks about integrals between bad points. First, we consider tiny integrals between
good points.

\begin{prop} \label{prop:complexitytinyint}
Let $P,Q \in X(\QQ_p)$ be good points that lie in the same residue disk and $\omega$ an element of our basis $[\omega_1,\ldots,\omega_{2g}]$. Then
Algorithm~\ref{alg:tiny} will compute $\int_{P}^Q \omega$ to $p$-adic precision~$N$ in time $\SoftOh(\log(p) d_x^2 d_y N^2)$.
\end{prop}
\begin{proof} 
Since $P$ is a good point, it can be written as $P = (x_0,y_0)$ with $x_0,y_0 \in \ZZ_p$, and we can take the local
coordinate at $P$ to be $t = x-x_0$. Suppose that we use $t$-adic precision $l$ in our computations, where $l$ will
be determined later. We need to expand $\omega$ as a power series in~$t$ using $t$-adic Hensel lifting in the 
ring $A = (\ZZ/p^N\ZZ)[t]/(t^l)$. Note that a single operation in $A$ takes time $\SoftOh(\log(p) N l)$. 

From the equation $Q(t+x_0,y(t)) = 0,$ which can be 
computed in $\BigOh(d_x d_y)$ operations in $A$ and has degree $d_x$ in $y$, we can compute $y(t)$ by Hensel lifting the 
solution $y_0$ modulo~$t$ in $\BigOh(d_x \log(l))$ operations in $A$. Computing the power series expansion of $1/r(x)$ in $t$ is similar but easier. 
By~\cite[Section 4.1]{tuitman:pc-general}, we have that 
$$\omega = g(x,y) \frac{dx}{r(x)}$$ 
where $g(x,y) \in \ZZ_p[x,y]$ is of degree at most $d_x-1$ in $y$ and degree $\BigOh(d_x d_y)$ in $x$. Therefore, $\omega(t)$ 
can be computed in $\BigOh(d_x^2 d_y \log(l))$ operations in~$A$, i.e. in time $\SoftOh(d_x^2 d_y N l)$. 

The actual integration and evaluation at the endpoints can be done naively in time $\SoftOh(\log(p) N l)$. 

Finally, by Remark~\ref{rem:l}, we should take $l$ minimal such that
\[
l+1 - \lfloor \log_p(l+1) \rfloor \geq N.
\]
Therefore $l$ is $\BigOh(N)$ and the proposition follows.
\end{proof}

Now we consider general integrals between good points.

\begin{prop}
Let $P,Q \in X(\QQ_p)$ be good points and $\omega$ an element of our basis $[\omega_1,\ldots,\omega_{2g}]$. 
Suppose that $\Phi$ is $p$-adically integral and $\ord_p(\det(\Phi-I)) = m$. Then Algorithm~\ref{alg:colemangoodi}  
will compute $\int_{P}^Q \omega$ to $p$-adic precision~$N-m$ in time $\SoftOh(p d_x^4 d_y^2 (N^2 + d_x d_y N))$.
\end{prop}

\begin{proof}
By Proposition~\ref{prop:complexityfrob}, the matrix $\Phi$ and the functions $f_1, \ldots, f_{2g}$ can be computed
to $p$-adic precision~$N$ in time $\SoftOh(p d_x^4 d_y^2 (N^2 + d_x d_y N))$.

Since $P$ is a good point, it can be written as $P = (x_0,y_0)$ with $x_0,y_0 \in \ZZ_p$. Note that 
$\Frob_p(P) = (x_0^p,y_p)$, where $y_p \in \ZZ_p$ can be obtained by Hensel lifting the solution 
$y_0^p$ modulo~$p$ to the equation $Q(x_0^p,y) = 0$. This can be done in $\BigOh(d_x d_y \log(N))$ 
operations in $\ZZ_p$, i.e. in time $\SoftOh(\log(p) d_x d_y \log^2(N))$. The complexity of computing
$\Frob_p(Q)$ is the same.

The tiny integrals $\int_P^{\Frob_p(P)} \omega_i$ and $\int_{\Frob_p(Q)}^{Q} \omega_i$ can 
be computed in time $\SoftOh(\log(p) d_x^2 d_y N^2)$ for a single value of $i$ by Proposition~\ref{prop:complexitytinyint}. 
Since $g$ is $\BigOh(d_x d_y)$ by~\cite[Proposition~4.1]{tuitman:pc-general}, we can do this for all $1 \leq i \leq 2g$ in time 
$\SoftOh(\log(p) d_x^3 d_y^2 N^2)$.

The functions $f_i$ have $\SoftOh(p d_x d_y N)$ terms, so can be evaluated at the points $P$ and $Q$ for all $1 \leq i \leq 2g$ 
in time $\SoftOh(p d_x^2 d_y^2 N^2)$. 

Finally, the $2g \times 2g$ linear system can be solved (naively) in time $\SoftOh(\log(p) d_x^3 d_y^3 N)$ and the proposition
follows.
\end{proof}

The input size of our algorithm is naturally determined by $d_x,d_y,N$ and $\log(p)$. Note that the complexity 
bounds above are polynomial in $d_x,d_y$ and $N$, but exponential in $\log(p)$. This is a typical feature of algorithms using $p$-adic 
cohomology, so should not come as a surprise. Actually, for integrals involving a finite bad point, the dependence of the 
complexity on~$p$ will even get a bit worse. By Remark~\ref{rem:nearbdprec}, we will have to compute in an extension of $\QQ_p$ 
of degree at least~$p$, which will worsen the dependence of the complexity on~$p$ from (quasi)linear to (quasi)quadratic. (Note that this does not happen at infinite points, so it might be useful to transform the curve so that a bad point of interest is 
moved to infinity, but we have not yet tried this.)

\subsection{Comparison with other algorithms}\label{subsec:comparison}\mbox{} \\

Another approach that has often been used to compute $\int_{P}^Q \omega$ is as follows. Let $J$ denote the Jacobian of $X$.  
First find some integer $k$ such that (the reduction mod~$p$ of) the point $k(P-Q)$ is trivial in $J(\FF_p)$. Note that for this~$k$, one would 
usually take the order of $J(\FF_p)$. After computing $k(P-Q)$ as an element (in the residue disk at~$0$) of $J(\QQ_p)$, one is reduced 
to computing a tiny integral over a divisor representing this point and dividing by~$k$. Here we discuss how this approach compares to ours.

Currently, implementations of algorithms to compute in $J(\QQ_p)$ are restricted to very special curves, e.g. hyperelliptic ones.
Indeed, for non-hyperelliptic curves of genus 4 or larger, there does not seem to be a readily available implementation of divisor arithmetic over $\QQ_p$.
In some cases, this can be circumvented by computing in $J(\QQ)$ instead, which then suffers from coefficient swell. However, 
even if the Jacobian arithmetic over $\QQ_p$ is not a problem, in general one has to compute the order of $J(\FF_p)$ first. Suppose 
that $p,N$ are small but $d_x$ or $d_y$ are large. The fastest known way to compute the order of $J(\FF_p)$
is then to compute the zeta function of $X \otimes \FF_p$ using the algorithm from~\cite{tuitman:pc-general} and evaluate its
numerator at~$1$. However, the complexity of that algorithm is that of our current algorithm with~$N$ of the order $\SoftOh(d_x d_y)$.
In other words, for $p,N$ fixed our algorithm computes Coleman integrals in time $\SoftOh(d_x^5 d_y^3)$, while the best known
algorithm for computing the order of $J(\FF_p)$ already takes time $\SoftOh(d_x^6 d_y^4)$.  In Section \ref{ex:challenge2}, we consider an example with large $d_x$ and $d_y$ and present some timings. 

As we will illustrate in the next section, the main strength of our algorithm is the range of examples it can routinely
handle.
\section{Examples}\label{sec:examples}
\subsection{An example from the work of Bruin--Poonen--Stoll} \mbox{ } \\

Let $X/\QQ$ be the genus 3 curve given by the following plane model: $$Q(x,y) = y^3 + (-x^2 - 1)y^2 - x^3y + x^3 + 2x^2 + x = 0.$$  Bruin, Poonen, and 
Stoll~\cite[Prop. 12.17]{bruin-poonen-stoll} show that, under the assumption of the Generalized Riemann Hypothesis, the Jacobian of $X$ has Mordell-Weil 
rank~1 over $\QQ$. (Note that our working plane model is given by taking the equation in  \cite[\S 12.9.2]{bruin-poonen-stoll}, provided by D. Simon, 
and setting $x:=1, z:=x$.)

We have $W^0 = I$, which means that $b^0 = [1,y,y^2]$  is an integral basis 
for the function field of $X$ over $\QQ[x]$. Moreover, we have
$$W^{\infty}=
\begin{pmatrix}
1 & 0 & 0 \\
0 & 1/x^2 & 0 \\
0 & -1/x, & 1/x^3
\end{pmatrix}
,$$
so that $b^{\infty}=[1,y/x^2,-y/x+y^2/x^3]$ is an integral basis 
for the function field of $X$ over $\QQ[1/x]$.

We consider the following points on $X$ : $P_1 = (0,0), P_2 = (0,1), P_3 = (-3,4), P_4 = (-1,0), P_5 = (-1,1)$, as well as three very infinite  points: 
$P_6$ with $b^{\infty}$-values $[1,0,1]$, $P_7$ with $b^{\infty}$-values $[1,1,1]$, and $P_8$ 
with $b^{\infty}$-values $[1,0,0]$.

In \cite[Prop. 12.17]{bruin-poonen-stoll}, the authors compute $X(\QQ)$ by using the fact that $[(P_3)-(P_2)]$ is of infinite order in $J(\QQ)$ and running a 3-adic Chabauty--Coleman argument.  
In particular, by computing $3$-adic tiny integrals between $P_2$ and $P_3$, they produce a two-dimensional subspace of regular $1$-forms 
annihilating rational points on $X$ and use the Coleman integrals of these differentials to show that these eight points are all of the 
rational points on $X$. 

Here we show how to produce a basis for the two-dimensional space of annihilating $1$-forms without immediately appealing to tiny integrals. While it is desirable to use tiny integrals 
whenever possible, some curves do not readily admit points of infinite order in $J(\QQ)$ that are given as small integral combinations of known rational points that 
allow a tiny integral computation. Consequently, in such a scenario, some arithmetic in the Jacobian would be needed to reduce the necessary 
Coleman integral computation to a tiny integral computation, as discussed in Section \ref{subsec:comparison}.  The computation below shows how one might bypass 
the Jacobian arithmetic by using Coleman integrals that are not necessarily tiny integrals. 

We have $r =  x(x+1)(x^8 + 7x^7 + 21x^6 + 31x^5 + 3x^4 - 51x^3 - 69x^2 - 23x + 4)$. Taking $p = 3$ makes all eight points  various types of bad:
\begin{center}
\begin{tabular}{| l |  c  | c  |c | }
\hline
Point $P$ & $r(x(P))$ & Type of point  \\ 
\hline
$P_1 = (0,0)$  & 0      & finite very bad\\
$P_2 = (0,1)$  & 0      & finite very bad\\
$P_3 = (-3,4)$ & $-600$ & finite bad\\
$P_4 = (-1,0)$ & 0      & finite very bad\\
$P_5 = (-1,1)$ & 0      & finite very bad \\
$1/x(P_6) = 0, b^{\infty} = [1,0,1] $ & $\infty$ & very infinite\\
$1/x(P_7) = 0, b^{\infty} = [1,1,1] $ & $\infty$ & very infinite \\
$1/x(P_8) = 0, b^{\infty} = [1,0,0] $ & $\infty$ & very infinite \\
\hline
\end{tabular}
\end{center}

We compute the $3$-adic Coleman integrals on a basis of $\Hrig^1(X \otimes \QQ_p)$, in particular, the regular $1$-forms 
are given by
\begin{align*}
\omega_1 &=  (b^0 \cdot (-8x^8 - 8x^7 + 86x^6 + 192x^5 + 118x^4 + 12x,\\
&\qquad\qquad -31x^7 - 98x^6 - 75x^5 + 70x^4 + 183x^3 + 234x^2 + 83x - 12, \\
&\qquad\qquad 31x^5 + 60x^4 - 52x^3 - 246x^2 - 119x + 12))\frac{dx}{r}, \\
\omega_2  &= (b^0 \cdot (2x^8 - 4x^7 - 56x^6 - 120x^5 - 76x^4 + 6x^2, \\
&\qquad\qquad13x^7 + 44x^6 + 45x^5 - 22x^4 - 81x^3 - 144x^2 - 77x + 12,\\
&\qquad\qquad -13x^5 - 24x^4 +  28x^3 + 138x^2 + 77x - 12))\frac{dx}{r},\\
\omega_3 &=  (b^0 \cdot (4x^7 + 22x^6 + 44x^5 + 30x^4 + 4x^3, \\
&\qquad\qquad  -3x^7 - 10x^6 - 11x^5 + 6x^4 + 19x^3 + 42x^2 + 27x - 4,\\
&\qquad\qquad  3x^5 + 4x^4 - 12x^3 - 46x^2 - 27x + 4))\frac{dx}{r},\end{align*}
producing the following values:
\begin{align*}
\int_{P_1}^{P_2} \omega_1 &= 2 \cdot 3^{2} + 3^{3} + 2 \cdot 3^{5} + 3^{6} + 2 \cdot 3^{7} + 3^{8} + O(3^{9}), \\
\int_{P_1}^{P_2} \omega_2 &= 3^{3} + 3^{4} + 2 \cdot 3^{5} + 2 \cdot 3^{6} + 3^{7} + O(3^{9}),\\
\int_{P_1}^{P_2} \omega_3 &= 3 + 2 \cdot 3^{2} + 3^{3} + 3^{4} + 3^{5} + O(3^{9}).
\end{align*}

We use the values of these three integrals (i.e., by computing the kernel of the associated $3 \times 1$ matrix) to compute that the two differentials
\begin{align*}
\xi_1 &= (1 + O(3^9))\omega_1 + O(3^9)\omega_2 + (430\cdot3 + O(3^9))\omega_3 \\
\xi_2 &=  O(3^9)\omega_1 + (1 + O(3^9))\omega_2 + (569\cdot3^2 + O(3^9))\omega_3
\end{align*}
give a basis for the regular $1$-forms annihilating rational points. Indeed, we can numerically see that the values of the two integrals 
$\int_{P_1}^P \xi_1, \int_{P_1}^P \xi_2$ vanish for all $P = P_3, P_4, \ldots, P_8$.  The code for this example can be found in the file~\texttt{./examples/bps.m} in~\cite{colemangit}.

\subsection{The modular curve $X_0(44)$} \mbox{ } \\

We consider the genus $4$ curve $X = X_0(44)$. We work with the plane model found by Yang \cite{yang}: $$Q(x,y) =y^5+12x^2y^3-14x^2y^2+(13x^4+6x^2)y-(11x^6+6x^4+x^2) = 0.$$ 
We have 
\[W^0=
\begin{bmatrix}
1 & 0  &  0 & 0  & 0 \\
0 & 1 &  0 & 0 & 0 \\
0 & 0 & 1 & 0 & 0 \\
0 & 0 & 0 & \frac{1}{x} & 0 \\
\frac{-10x^3}{x^4 + 6x^2 + 1} &  \frac{-6x^3 - 13x}{x^4 + 6x^2 + 1} &  \frac{x^3 + 12x}{x^4 + 6x^2 + 1} &  \frac{-x}{x^4 + 6x^2 + 1} & \frac{1}{x^5 + 6x^3 + x} \end{bmatrix}.\]

\noindent Indeed, this plane model is singular, as we see $W^0 \neq I$.  We have $$r = x(x^4 + 6x^2 + 1)(45753125x^8 + 8440476x^6 + 1340814x^4 + 69756x^2 + 3125)$$
and
$$b^0 = \left[1, y, y^2, \frac{y^3}{x}, \frac{-10x^4-(6x^4-13x^2)y + (x^4+12x^2)y^2 - x^2y^3+1}{x^5 + 6x^3 +x}\right].$$
We have that $(Q,W^0, W^{\infty})$ has good reduction at $p=7$.
Let $P_1$ be the (good) point $(1,1)$ and consider the unique point $P_2$ on the smooth model which lies over the singularity $x=0,y=0$ of the plane model.
At it turns out, at $P_2$ we have that $y^3/x$ is a local coordinate and the values of the $b^0$ are $[1,0,0,0,0]$.  
Computing $7$-adic integrals gives 
$$\left(\int_{P_1}^{P_2} \omega_1, \int_{P_1}^{P_2} \omega_2, \int_{P_1}^{P_2} \omega_3, \int_{P_1}^{P_2} \omega_4\right) = (O(7^9),O(7^9),O(7^9),O(7^9)),$$ 

\noindent which seems to suggest that $[(P_2)-(P_1)]$ is a torsion point in the Jacobian of $X$. A computation in Magma verifies that $15[(P_2)-(P_1)] = 0.$
The code for this example can be found in the file \texttt{./examples/X0\_44.m} in \cite{colemangit}.

\subsection{A superelliptic genus 4 curve} \mbox{ } \\

We consider the superelliptic genus $4$ curve $X/\QQ$ given by the plane model $$Q(x,y) = y^3 - (x^5 - 2x^4 - 2x^3 - 2x^2 - 3x) = 0.$$ 
Using the \texttt{Magma} intrinsic \texttt{RankBounds}, which is based on~\cite{poonenschaefer} and implemented by Creutz, we find that the Mordell-Weil rank of its Jacobian is $1$. A 
search yields the rational points 
$$P_1=(1,-2), P_2=(0,0), P_3=(-1,0), P_4=(3,0), P_5=\infty.$$
We have $b^0=[1,y,y^2]$ and $r=x^5 - 2x^4 - 2x^3 - 2x^2 - 3x$. A basis for the regular $1$-forms on $X$ is given by 
\[
\omega_1 = \frac{y dx}{r}, \; \; \; \omega_2 = \frac{xy dx}{r}, \;\;\; \omega_3 = \frac{x^2y dx}{r}, \;\;\; \omega_4=\frac{y^2 dx}{r}. 
\]
Now we take $p=7$ and compute
\[
\int_{P_1}^{P_2} \omega_1 =  12586493\cdot7 + O(7^{10}).
\]
Since this integral does not vanish, $[(P_2) - (P_1)]$ is non-torsion in the Jacobian. 

The space of annihilating regular $1$-forms is $3$-dimensional, and a basis is given by
\begin{align*}
\xi_1 &= (1 + O(7^{10})) \omega_1 + O(7^{10}) \omega_2 + O(7^{10}) \omega_3 - (139167240 + O(7^{10})) \omega_4 \\
\xi_2 &= O(7^{10}) \omega_1 + (1 + O(7^{10}))\omega_2 + O(7^{10}) \omega_3 + (93159229 + O(7^{10})) \omega_4\\
\xi_3 &= O(7^{10}) \omega_1 + O(7^{10}) \omega_2 + (1 + O(7^{10})) \omega_3 + (8834289 + O(7^{10})) \omega_4.
\end{align*}
Indeed, we can numerically see that the values of the $3$ integrals 
$\int_{P_1}^P \xi_1, \int_{P_1}^P \xi_2, \int_{P_1}^P \xi_3$ vanish for $P = P_3, P_4,P_5$.  
The code for this example can be found in the file \texttt{./examples/C35.m} in \cite{colemangit}. 

\subsection{A curve of genus 55}\label{ex:challenge2}
As discussed in Section \ref{subsec:comparison}, to compute Coleman integrals using the other leading approach, there are challenges to working in the Jacobian of the curve in the case of large genus. Here we present some timings indicating the feasibility of our algorithm. The computations in this subsection were carried out on a single core of a 28-core 2.2 GHz Intel Xeon server with 256GB RAM.

Here  we consider the genus 55 curve $X$ with plane model given by $Q(x,y)= 0$ below:
\small{\begin{align*}
Q(x,y)&= x^{11}y - x^7y^5 - x^6y^6 - x^4y^8 + xy^{11} + y^{12} + x^{11} - x^{10}y + x^8y^3 - x^6y^5 + x^5y^6 + x^3y^8 - x^2y^9 - xy^{10} +\\
&\quad y^{11} + x^{10} + x^9y - x^8y^2 + x^7y^3 + x^6y^4 + x^5y^5 - x^4y^6 + xy^9 +y^{10} - x^9 + x^8y+ x^7y^2 + x^6y^3 + x^5y^4 + \\
&\quad  x^4y^5 + x^3y^6 - x^2y^7 + y^9 + x^8 - x^7y + x^6y^2 - x^5y^3 + xy^7 + y^8 + x^7 + x^6y +x^5y^2 - x^2y^5 - xy^6 +  \\
&\quad  y^7 - x^6 - x^4y^2 - x^2y^4 + xy^5 - x^5 + x^3y^2 - x^2y^3 + y^5 - x^4 + x^3y + x^2y^2 + xy^3 + y^4 - x^2y - xy^2 + \\
&\quad y^3 - x^2 - xy + x + y.\end{align*}}
\normalsize
This example was constructed using the \texttt{Magma} intrinsic \texttt{RandomPlaneCurve}, with the call
\begin{verbatim}
> P<x,y,z>:=ProjectiveSpace(Rationals(),2);
> RandomPlaneCurve(12,[0],P:RandomBound:=1);
\end{verbatim}
producing  a smooth plane curve of degree 12 and coefficients randomly selected from $\{-1,0,1\}$. We generated a number of such curves and considered a selection that had at least 3 rational points. We present one illustrative example here.

Let $p=7$ and consider $P_1 = (0,0)$ and $P_2 = (1,0)$, which are each good points on $X$. We compute the Coleman integrals $\left\{\int_{P_1}^{P_2} \omega_i\right\}_{i = 1}^{110}$ for the basis $\{\omega_i\}$ of $H^1_{\textrm{rig}}(X\otimes \QQ_p)$ constructed as in Definition \ref{def:basis} with $N = 5$ as our precision. We find that $$\int_{P_1}^{P_2} \omega_1 = 5 \cdot 7 + O(7^2),$$ and we deduce that the Jacobian of $X$ has positive rank.

The computation of \texttt{coleman\_data} took 79685 s, after which the call to \texttt{coleman\_integrals\_on\_basis} took 39 s.
   
The code for this example can be found in the file~\texttt{./examples/g55.m} in~\cite{colemangit}.

Further examples illustrating how to call and use the code are available in the file \texttt{examples.pdf} in \cite{colemangit}.

\section*{Acknowledgements}
We would like
to thank Amnon Besser, Netan Dogra, Alan Lauder, Steffen M\"{u}ller,
and Floris Vermeulen for helpful discussions, as well as the anonymous
referees for their valuable comments on earlier versions of this
manuscript.
Balakrishnan is supported in part by NSF grant DMS-1702196, the Clare Boothe Luce Professorship (Henry Luce Foundation), and Simons Foundation grant \#550023.
Tuitman is a Postdoctoral Researcher of the Fund for Scientific Research FWO - Vlaanderen. \\

\bibliographystyle{amsalpha} 
\providecommand{\bysame}{\leavevmode\hbox to3em{\hrulefill}\thinspace}
\providecommand{\MR}{\relax\ifhmode\unskip\space\fi MR }
\providecommand{\MRhref}[2]{%
  \href{http://www.ams.org/mathscinet-getitem?mr=#1}{#2}
}
\providecommand{\href}[2]{#2}

\end{document}